\pdfoutput=1
\documentclass[a4paper,12pt]{amsart}
\usepackage[margin=2.5cm]{geometry}
\usepackage[utf8]{inputenc}
\usepackage[T1]{fontenc}
\usepackage[charter]{mathdesign}
\usepackage{microtype}
\usepackage{hyperref}
\hypersetup{
colorlinks=true,
citecolor=[rgb]{0,0,0.75},
linkcolor=[rgb]{0,0,0.75},
urlcolor=blue,
bookmarks=true,
pdfpagemode=UseNone,
pdfstartview=FitH,
pdfdisplaydoctitle=true,
pdfborder={0 0 0}, 
pdftitle={Recurrence for non-commuting transformations},
pdfauthor={Nikos Frantzikinakis and Pavel Zorin-Kranich},
pdfsubject={Mathematics Subject Classification (2010): 37A15},
pdfkeywords={multiple recurrence, non-commuting transformations},
pdflang=en-US
}
\usepackage[safeinputenc=true,backref=true,style=alphabetic,giveninits,maxnames=4,useprefix=true,doi=false,eprint=true,isbn=false,url=false]{biblatex}

\newbibmacro{string+doiurl}[1]{%
  \iffieldundef{doi}{%
    \iffieldundef{url}{%
         #1%
      }{%
      \href{\thefield{url}}{#1}%
    }%
  }{%
    \href{http://dx.doi.org/\thefield{doi}}{#1}%
  }%
}
\DeclareFieldFormat[article,misc]{title}{\usebibmacro{string+doiurl}{\mkbibquote{#1}}}
\DefineBibliographyStrings{english}{%
  backrefpage = {$\uparrow$},%
  backrefpages = {$\uparrow$}%
}

\makeatletter
\def\woMR#1{\w@MR#1MR#1MR\relax}%
\def\w@MR#1MR#2MR#3\relax{#2}

\def\MR@URL#1 #2\relax{http://www.ams.org/mathscinet-getitem?mr=#1}
\def\MRnumber#1{\href{\MR@URL#1 \relax}{\nolinkurl{\woMR#1}}}
\makeatother

\DeclareFieldFormat{mrnumber}{\mkbibacro{MR}\addcolon\space\MRnumber{#1}}
\newbibmacro*{doi+eprint+url}{%
  \iftoggle{bbx:doi}
    {\printfield{doi}}
    {}%
  \newunit\newblock
  \iftoggle{bbx:eprint}
    {\usebibmacro{eprint}}
    {}%
  \newunit\newblock
  \iftoggle{bbx:url}
    {\usebibmacro{url+urldate}}
    {}%
  \newunit\newblock
  \printfield{mrnumber}}

\numberwithin{equation}{section}
\newtheorem{theorem}[equation]{Theorem}
\newtheorem{lemma}[equation]{Lemma}
\newtheorem{proposition}[equation]{Proposition}
\newtheorem{corollary}[equation]{Corollary}
\theoremstyle{definition}

\theoremstyle{remark}

\newcommand*{\N}{\mathbb{N}}
\newcommand*{\Z}{\mathbb{Z}}

\newcommand*{\Q}{\mathbb{Q}}
\newcommand*{\T}{\mathbb{T}}
\newcommand*{\dif}{\mathrm{d}}

\newcommand*{\E}{\mathbb{E}}

\newcommand{\ud}{\overline{d}}

\newcommand*{\K}{\mathcal{K}}
\newcommand*{\Krat}{\mathcal{K}_{\mathrm{rat}}}

\begin{document}
\subjclass[2010]{37A15}
\title[Recurrence for non-commuting transformations]{Multiple recurrence for non-commuting transformations along rationally independent polynomials}
\author{Nikos Frantzikinakis}
\author{Pavel Zorin-Kranich}
\keywords{Multiple recurrence, non-commuting transformations}
\begin{abstract}
We prove a multiple recurrence result for arbitrary measure-preserving transformations along polynomials in two variables of the form $m+p_{i}(n)$, with rationally independent $p_{i}$'s with zero constant term.
This is in contrast to the single variable case, in which even double recurrence fails unless the transformations generate a virtually nilpotent group.
The proof involves reduction to nilfactors and an equidistribution result on nilmanifolds.
\end{abstract}
\maketitle

\section{Introduction}
The polynomial Szemer\'edi theorem \cite{MR1325795} for commuting invertible measure-preserving transformations $T_{1},\dots,T_{\ell}$ on a probability space $(X,\mu)$ asserts that for every positive measure set $A\subset X$ and any integer polynomials $p_{i}:\Z^{d}\to\Z$ with zero constant term the set of $n\in\Z^{d}$ such that
\begin{equation}
\label{eq:mucap}
\mu(A\cap T_{1}^{p_{1}(n)}A \cap \dots \cap T_{\ell}^{p_{\ell}(n)}A) > 0
\end{equation}
has positive lower density.
This result extends to nilpotent groups of transformations \cite{MR1650102} but fails quite dramatically for groups that are not virtually nilpotent.
Furstenberg constructed weakly mixing measure-preserving transformations $T,S$ on a probability space $(X,\mu)$ such that for some positive measure subset $A\subset X$ we have $\mu(T^{n}A \cap S^{n}A) = 0$ for every $n\in\N$ \cite[p.\ 40]{MR603625}.
His construction has been extended by the first author, Lesigne, and Wierdl to obtain $\mu(T^{a(n)}A \cap S^{b(n)}A) = 0$ for any given injective sequences $a,b : \N \to \Z\setminus\{0\}$, e.g.\ $(n)$ and $(n^{2})$ \cite[Theorem 1.7]{MR3043589}.
The transformations in Furstenberg's example generate a solvable group that is not virtually nilpotent, and it has been shown by Bergelson and Leibman that every such group admits a measure-preserving action for which recurrence fails for some pair of elements in the group \cite{MR2041260}.

Thus it may come as a surprise that we obtain a multiple recurrence result without any algebraic assumptions on the measure-preserving transformations.
\begin{theorem}
\label{thm:recurrence}
Let $T_{1},\dots,T_{\ell}$ be invertible  measure-preserving transformations on a probability space $(X,\mu)$ and $p_{1},\dots,p_{\ell}\in\Z[n]$ be rationally independent polynomials with zero constant term.
Then for every set of positive measure $A\subset X$ and every $\epsilon>0$ the set of pairs $(m,n)$ such that
\[
\mu(A\cap T_{1}^{m+p_{1}(n)}A \cap\dots\cap T_{\ell}^{m+p_{\ell}(n)}A)
> \mu(A)^{\ell+1} - \epsilon
\]
has positive lower density with respect to any F\o{}lner sequence of the form $\Phi_{N}=[1,N]\times [1,b(N)]$ for a sequence $b:\N\to\N$ such that $b(N)\to\infty$ and $b(N)/N^{1/d}\to 0$ as $N\to\infty$, where $d=\max_{i}\deg p_{i}$.
\end{theorem}
This gives a partial answer to a problem from \cite{MR2995648}.
The question whether a similar recurrence result (without the lower bound) for non-commuting transformations holds for rationally dependent polynomials remains open.
In particular, even the case $\ell=3$, $p_{1}(n)=0$, $p_{2}(n)=n$, $p_{3}(n)=2n$ is open.

The lower bound $\mu(A)^{\ell+1}$ is optimal (as can be seen considering weakly mixing transformations and using Proposition~\ref{prop:Krat-char}) and new even for commuting transformations.
We note that in the case of a not necessarily ergodic single transformation $T_{1}=\dots=T_{\ell}=T$ and without the extra variable $m$ the corresponding bound for \eqref{eq:mucap} is known for families of linearly independent polynomials \cite{MR2254556}.
If $T$ is ergodic then there is also a lower bound for \eqref{eq:mucap} for the families $(n,2n)$, $(n,2n,3n)$ \cite{MR2138068}, and some other families of polynomials \cite[Theorem C]{MR2415080}.
This result fails for $(n,2n)$ without the ergodicity assumption and for $(n,2n,3n,4n)$ even with the ergodicity assumption \cite{MR2138068}.

For commuting transformations $T_{1},\dots,T_{\ell}$ the corresponding lower bounds are known for families similar to $(n,n^{2},\dots,n^{\ell})$ without ergodicity assumptions \cite{MR2795725} and the family $(n,n)$ if $T_{1}$ and $T_{2}$ are jointly ergodic \cite{MR2794947}, although in the latter case the optimal bound is $\mu(A)^{4}$ and not $\mu(A)^{3}$ as one would expect from the weakly mixing case.

By the Furstenberg correspondence principle (see e.g.\ \cite[Proposition 4.1]{MR2354320} for an appropriate version), Theorem~\ref{thm:recurrence} has the following combinatorial consequence.
\begin{corollary}
Let $G$ be a countable amenable group with a fixed right F\o{}lner sequence and let $\ud$ denote either the upper density or the upper Banach density on $G$.
Let also $g_{1},\dots,g_{\ell}\in G$ and $p_{1},\dots,p_{\ell}\in \Z[n]$ be rationally independent polynomials with zero constant term.
Then for every set $E\subset G$ with $\ud(E)>0$ and every $\epsilon>0$ the set of pairs $(m,n)$ such that
\[
\ud(E\cap Eg_{1}^{m+p_{1}(n)} \cap\dots\cap Eg_{\ell}^{m+p_{\ell}(n)})
>
\ud(E)^{\ell+1} - \epsilon
\]
has positive lower density with respect to any F\o{}lner sequence as in Theorem~\ref{thm:recurrence}.
\end{corollary}
We note that throughout this article the assumption that the polynomials vanish at zero can be relaxed to joint intersectivity, see \cite{MR2435427} for the definition of this property.

\section{An equidistribution result on nilmanifolds}
A ($k$-step) \emph{nilmanifold} is a quotient space $X=G/\Gamma$, where $G$ is a ($k$-step) nilpotent Lie group and $\Gamma\leq G$ is a discrete cocompact subgroup, with the unique $G$-invariant probability measure.
A \emph{nilsystem} $(X,T)$ is a nilmanifold $X=G/\Gamma$ with a map of the form $Tx=ax$, $a\in G$.
It is known that every ergodic nilsystem is uniquely ergodic, see e.g.\ \cite[\textsection 2.19]{MR2122919}.
A \emph{($k$-step) basic nilsequence} is a sequence of the form $(f(T^{n}x))$, where $(X,T)$ is a ($k$-step) nilsystem, $x\in X$, and $f\in C(X)$.
Without loss of generality the nilsystem in the definition of a basic nilsequence can be taken to be ergodic since every nilsystem is a disjoint union of ergodic nilsystems by \cite[Remark 2.22]{MR2122919}.
A \emph{$k$-step nilsequence} is a uniform limit of $k$-step basic nilsequences.

Recall that a map $g : \Z^{d} \to X$ to a topological space $X$ with a Borel measure $\mu$ is said to be \emph{well-distributed} on $(X,\mu)$ if for every F\o{}lner sequence $(\Phi_{N})$ in $\Z^{d}$ and every continuous function $f\in C(X)$ we have $\E_{n\in \Phi_{N}} f(g(n)) \to \int_{X} f \dif\mu$ as $N\to\infty$ (here and later we denote averages by $\E_{n\in\Phi}=\frac{1}{|\Phi|}\sum_{n\in\Phi}$).
The basic well-distribution criterion for nilmanifolds is due to Leibman.
\begin{lemma}[{\cite[\textsection 1.10]{MR2122920}}]
\label{lem:equid-mod-comm}
Let $X=G/\Gamma$ be a nilmanifold, $a_{i}\in G$, $p_{i}:\Z^{d}\to\Z$ be polynomials, $i=1,\dots,\ell$, and $x\in X$.
Then the following are equivalent.
\begin{enumerate}
\item The sequence $(a_{1}^{p_{1}(n)}\dots a_{\ell}^{p_{\ell}(n)}x)_{n\in\Z^{d}}$ is well-distributed on $X$,
\item The sequence $(a_{1}^{p_{1}(n)}\dots a_{\ell}^{p_{\ell}(n)}x)_{n\in\Z^{d}}$ is dense in $X$,
\item\label{emc:quot-well} The sequence $([G^{o},G^{o}]a_{1}^{p_{1}(n)}\dots a_{\ell}^{p_{\ell}(n)}x)_{n\in\Z^{d}}$ is well-distributed on $[G^{o},G^{o}] \backslash X$ (here $G^{o}$ denotes the connected component of the identity in a Lie group $G$),
\item\label{emc:quot-dense} The sequence $([G^{o},G^{o}]a_{1}^{p_{1}(n)}\dots a_{\ell}^{p_{\ell}(n)}x)_{n\in\Z^{d}}$ is dense in $[G^{o},G^{o}] \backslash X$.
\end{enumerate}
\end{lemma}
The next result tells us that every ergodic nilsystem can be decomposed into finitely many totally ergodic nilsystems.
\begin{lemma}[{\cite[Proposition 2.1]{MR2415080}}]
\label{lem:conn-comp-tot-erg}
Let $(X,T)$ be an ergodic nilsystem.
Then
\begin{enumerate}
\item $(X,T)$ is totally ergodic if and only if $X$ is connected, and
\item there exists $r\in\N$ such that $(X',T^{r})$ is an ergodic nilsystem for every connected component $X'$ of $X$, and in particular is totally ergodic.
\end{enumerate}
\end{lemma}
Finally, the systems arising in Lemma~\ref{lem:equid-mod-comm}(\ref{emc:quot-well},\ref{emc:quot-dense}) in the totally ergodic case have a particularly simple algebraic structure.
\begin{lemma}[{\cite[Proposition 3.1]{MR2191231}}]
\label{lem:unipotent}
Let $(X,T) = (G/\Gamma,T)$ be a totally ergodic nilsystem and assume that $G^{o}$ is commutative.
Then $(X,T)$ is topologically isomorphic to a unipotent affine transformation on a torus, i.e.\ a transformation of the form $(\T^{d},S)$, where $Sx=x+Nx+b$ with a nilpotent integer matrix $N$ and $b\in\T^{d}$.
\end{lemma}
With these tools at hand we can show our equidistribution result on nilmanifolds.
\begin{proposition}
\label{prop:equid}
Let $(X_{i},T_{i})$, $i=1,\dots,\ell$, be totally ergodic nilsystems and $p_{i}$, $i=1,\dots,\ell$, integer polynomials rationally independent from $1$.
Then for every $x_{1},\dots,x_{\ell}$ the polynomial sequence $(T_{1}^{m+p_{1}(n)}x_{1},\dots,T_{\ell}^{m+p_{\ell}(n)}x_{\ell})_{(m,n)\in\Z^{2}}$ is well-distributed on $X_{1}\times\dots\times X_{\ell}$.
\end{proposition}
It is crucial for our argument that the above equidistribution property holds for \emph{every} tuple $x_{1},\dots,x_{\ell}$.
The following example shows that it is not possible to remove the extra variable $m$ from the proposition.
Let $X=\T^{2}$ and $T : X\to X$ be defined by $(x,y) \mapsto (x+\alpha, y+2x+\alpha)$ with $\alpha$ irrational.
Then $(X,T)$ is topologically isomorphic to a totally ergodic nilsystem, but the sequence $(T^{n}(0,0),T^{n^{2}}(0,0)) = (n\alpha, n^{2}\alpha, n^{2}\alpha, n^{4}\alpha)$ is not equidistributed on $X\times X$.
\begin{proof}
Suppose that $X_{i}=G_{i}/\Gamma_{i}$, $i=1,\dots,\ell$, where $G_{i}$ are nilpotent Lie groups and $\Gamma_{i}\leq G_{i}$ discrete cocompact subgroups, and $T_{i}g\Gamma_{i} = a_{i} g \Gamma_{i}$ for some $a_{i}\in G_{i}$.
Then $X_{1}\times\dots\times X_{\ell} = (G_{1}\times\dots\times G_{\ell})/(\Gamma_{1}\times\dots\times \Gamma_{\ell})$.
By Lemma~\ref{lem:equid-mod-comm} we may assume that $[(G_{1}\times\dots\times G_{\ell})^{o},(G_{1}\times\dots\times G_{\ell})^{o}] = [G_{1}^{o},G_{1}^{o}] \times\dots\times [G_{\ell}^{o},G_{\ell}^{o}]$ is trivial, i.e.\ the connected components of the origins $G_{i}^{o}$ are commutative groups.
Thus by Lemma~\ref{lem:unipotent} we may assume that $(X_{i},T_{i})$ are unipotent affine transformations, i.e.\ $X_{i}=\T^{d_{i}}$ and $T_{i}x=x+N_{i}x+b_{i}$ with a nilpotent integer matrix $N_{i}$ and $b_{i} \in \T^{d_{i}}$.
Then by induction on $n$ we have
\[
T_{i}^{n}x_{i} = x_{i} + \sum_{r=0}^{d_{i}-1} \binom{n}{r+1} N_{i}^{r} (N_{i}x_{i}+b_{i}).
\]
Fix $(x_{1},\dots,x_{\ell})$ and assume that $(T_{1}^{m+p_{1}(n)}x_{1},\dots,T_{\ell}^{m+p_{\ell}(n)}x_{\ell})$ is not well-distributed on the torus $X_{1}\times\dots\times X_{\ell}$.
By Weyl's equidistribution criterion we obtain $l_{i}\in\Z^{d_{i}}$, not all of which are zero, such that
\[
\sum_{i=1}^{\ell} l_{i} \cdot \sum_{r=0}^{d_{i}-1} \binom{m+p_{i}(n)}{r+1} N_{i}^{r} (N_{i}x_{i}+b_{i})
\equiv
0
\mod \Q[m,n].
\]
Let $\alpha_{i,r} := l_{i} \cdot N_{i}^{r}(N_{i}x_{i}+b_{i})$, we will show that $\alpha_{i,r}\in\Q$ for all $i,r$.
Then Weyl's equidistribution criterion implies that $T_{i}^{n}x_{i}$ is not well-distributed on $X_{i}$ for those $i$ with $l_{i}\neq 0$, thus contradicting unique ergodicity of $(X_{i},T_{i})$.

Changing the order of summation we obtain
\[
\sum_{r=0}^{R} \sum_{i=1}^{\ell} \binom{m+p_{i}(n)}{r+1} \alpha_{i,r}
\equiv
0
\mod \Q[m,n],
\]
where $R = \max_{i} d_{i}-1$.
Recall that the binomial coefficient $\binom{n}{k}$ is a polynomial in $n$ of degree $k$.

With this in mind we consider the $m^{R}$ term in the above sum.
The contributions to that term come from $r=R-1$ (which gives a scalar multiple of $m^{R}$) and from $r=R$ which gives $\frac{1}{R!} \sum_{i}m^{R} p_{i}(n) \alpha_{i,R}$ plus some scalar multiple of $m^{R}$.
Since the polynomials $p_{i}$ are rationally independent from $1$ this implies $\alpha_{i,R}\in\Q$.
Hence we obtain
\[
\sum_{r=0}^{R-1} \sum_{i=1}^{\ell} \binom{m+p_{i}(n)}{r+1} \alpha_{i,r}
\equiv
0
\mod \Q[m,n],
\]
and we conclude by induction on $R$.
\end{proof}

\section{A decomposition result}
Let $(X,\mu,T)$ be an arbitrary measure-preserving system (not necessarily ergodic).
The \emph{uniformity seminorms} $U^{k}$, $k=1,2,\dots$, are defined inductively by
\[
\|f\|_{U^0(X,\mu,T)}:=\int_X f\dif\mu,\quad
\|f\|_{U^{k+1}(X,\mu,T)}^{2^{k+1}}:=\lim_{N\to\infty}\E_{n\in [1,N]} \|T^n f \bar{f}\|_{U^{k}(X,\mu,T)}^{2^{k}},\quad
f\in L^{\infty}(X).
\]
We will use the following decomposition result for functions on a not necessarily ergodic system.
For ergodic systems it is a direct consequence of the Host-Kra structure theorem \cite{MR2150389}.
\begin{theorem}[{\cite[Proposition 3.8]{MR2995648}}]
\label{thm:structure}
Let $(X,\mu,T)$ be a measure-preserving system, $f\in L^\infty(X)$, and $k\in \N$.
Then for every $\epsilon>0$ there exist measurable functions $f^s$, $f^u$, and $f^e$, with $L^\infty$ norm at most $2\|f \|_{\infty}$, such that
\begin{enumerate}
\item $f=f^s+f^u+f^e$,
\item $\|f^u\|_{U^{k+1}}=0$, $\|f^e\|_{1} < \epsilon$, and
\item for every $x\in X$ the sequence $(f^s (T^nx))_{n\in\N}$ is a $k$-step nilsequence.
\end{enumerate}
\end{theorem}
Recall that the \emph{rational Kronecker factor} $\Krat(X,T)$ is defined by $\Krat(X,T) = \vee_{r>0}\K_{r}(X,T)$, where $\K_{r}(X,T)$ is the $T^{r}$-invariant sub-$\sigma$-algebra.
The martingale convergence theorem and the pointwise ergodic theorem show that
\begin{equation}
\label{eq:EKrat}
\E_{\mu}(f|\Krat(X,T))(x)
=
\lim_{r\to\infty} \E_{\mu}(f|\K_{r!}(X,T))(x)
=
\lim_{r\to\infty} \lim_{N\to\infty} \E_{n\leq N} f(T^{r! \cdot n} x)
\end{equation}
for every bounded function $f$ and $\mu$-a.e.\ $x$.
We will use a version of this conditional expectation operator for nilsequences.
\begin{lemma}
\label{lem:limit-basic-nilseq-integral}
For every nilsequence $a=(a_{n})$ define a sequence $P(a) = (P_{k}(a))$ by
\begin{equation}
\label{eq:perp-Krat-pointwise}
P_{k}(a) := \lim_{r\to\infty} \lim_{N\to\infty} \E_{n\leq N} a_{r! \cdot n+k}.
\end{equation}
Then the following statements hold.
\begin{enumerate}
\item\label{PKrat:integral} If $a_{n} = g(S^{n}y)$ is a basic nilsequence with an ergodic nilsystem $(Y,S)$, $g\in C(Y)$, and $y\in Y$, then $P_{k}(a) = \int_{Y'_{k}} g$, where $Y'_{k}$ is the connected component of $S^{k}y$ in $Y$ and the integral is taken with respect to the unique normalized Haar measure on $Y'_{k}$.
\item\label{PKrat:proj} $P$ is a well-defined contractive linear projection on the vector space of all nilsequences with the $\ell^{\infty}$ norm (in particular the limit \eqref{eq:perp-Krat-pointwise} exists) that leaves the subspace of basic nilsequences invariant.
\item\label{PKrat:ker} Basic nilsequences are $\ell^{\infty}$-dense in $\ker P$.
\end{enumerate}
\end{lemma}
\begin{proof}
To show (\ref{PKrat:integral}) let, with the above notation, $r_{0}$ be so large that $(Y_{k}',S^{r!})$ is ergodic for every $r>r_{0}$ (such $r_{0}$ exists by Lemma~\ref{lem:conn-comp-tot-erg}).
Recall that the nilsystem $(Y_{k}',S^{r!})$ is uniquely ergodic, so that
\[
\lim_{N\to\infty} \E_{n\leq N} a_{r! \cdot n+k}
=
\lim_{N\to\infty} \E_{n\leq N} g((S^{r!})^{n} S^{k}y)
=
\int_{Y_{k}'} g,
\]
and the claim follows.
In particular, the limit \eqref{eq:perp-Krat-pointwise} exists for every basic nilsequence, and by density for every nilsequence.

It is clear that $P$ is linear and contractive.
Let $a$ be a basic nilsequence, then $P(a)$ is also a basic nilsequence as witnessed, with the notation of (\ref{PKrat:integral}), by the nilsystem $(Y,S)$ and the continuous function that on each connected component of $Y$ equals the mean value of $g$ over this connected component.
In particular, $P^{2}(a) = P(a)$ for every basic nilsequence $a$ and we obtain (\ref{PKrat:proj}) by density.

Finally, if $a \in \ker P$ is a nilsequence and $a'$ is a basic nilsequence such that $\|a-a'\|_{\ell^{\infty}} < \epsilon$, then $a'-P(a') \in\ker P$ is also a basic nilsequence and $\|a-(a'-P(a'))\|_{\ell^{\infty}} \leq \|a-a'\|_{\ell^{\infty}} + \|P(a)-P(a')\|_{\ell^{\infty}} < 2\epsilon$, and (\ref{PKrat:ker}) follows.
\end{proof}
For functions that are orthogonal to the rational Kronecker factor we have the following version of Theorem~\ref{thm:structure}.
\begin{corollary}
\label{cor:structure-Krat}
If under the assumptions of Theorem~\ref{thm:structure} we also have $\E_{\mu}(f|\Krat(X,T)) = 0$, then, relaxing the $L^{\infty}$ bound to $4\|f\|_{\infty}$ and the $L^{1}$ bound to $2\epsilon$, we may assume that $P(f^{s}(T^{n} x)) = 0$ for every $x$.
\end{corollary}
\begin{proof}
Consider the decomposition $f=f^{s}+f^{u}+f^{e}$ given by Theorem~\ref{thm:structure}.
It is known that $\|f^{u}\|_{U^{k+1}(X,\mu,T)}=0$, $k\geq 1$, implies that $\|f^{u}\|_{U^{1}(X,\mu,T^{r!})}=0$ for every $r\in\N$, see e.g.\ \cite[\textsection 2.2]{MR2795725}.
By the mean ergodic theorem we have $\|f^{u}\|_{U^{1}(X,\mu,T^{r!})} = \|\E(f^{u}|\K_{r!}(T))\|_{L^{2}(X,\mu)}$, so that $\E_{\mu}(f^{u}|\Krat(X,T)) = 0$ by \eqref{eq:EKrat}.
By linearity of conditional expectation we obtain
\begin{equation}
\label{eq:Efs-small}
\|\E_{\mu}(f^{s}|\Krat(X,T))\|_{1} = \|\E_{\mu}(f^{e}|\Krat(X,T))\|_{1} < \epsilon.
\end{equation}
By Lemma~\ref{lem:limit-basic-nilseq-integral} the limit
\[
f'(x):=P_{0}(f^{s}(T^{n}x))=\lim_{r\to\infty} \lim_{N\to\infty} \E_{n\leq N} f^{s}(T^{r! \cdot n} x)
\]
exists for every $x$.
The function $f'$ is measurable and bounded by $2\|f\|_{\infty}$.
On the other hand by \eqref{eq:EKrat} we have $f'=\E_{\mu}(f^{s}|\Krat(X,T))$, so that $\|f'\|_{1} < \epsilon$ by \eqref{eq:Efs-small}.
The claim follows from Lemma~\ref{lem:limit-basic-nilseq-integral}(\ref{PKrat:proj}) upon replacing $f^{s}$ by $f^{s}-f'$ and $f^{e}$ by $f^{e}+f'$.
\end{proof}

\section{Characteristic factors and proof of the recurrence result}
In order to prove Theorem~\ref{thm:recurrence} we show that the rational Kronecker factors of $T_{1},\dots,T_{\ell}$ are characteristic for the (pointwise) convergence of the corresponding multiple ergodic averages, from which point the recurrence result follows by a computation.

In order to dispose of the uniform part in the decomposition provided by Theorem~\ref{thm:structure} we use the following result.
\begin{theorem}[{\cite[Theorem 1.4]{MR2995648}}]
\label{thm:Zk-char}
Let $(X,\mu),T_{1},\dots,T_{\ell},(b(N))$ be as in Theorem~\ref{thm:recurrence}
and let $p_{1},\dots,p_{\ell}$ be essentially distinct polynomials (i.e.\ their pairwise differences are not constant).
Then there exists $k$ such that for any functions $f_{i}\in L^{\infty}(X)$, $i=1,\dots,\ell$, with $\|f_{i_{0}}\|_{U^{k}(X,\mu,T_{i_{0}})} = 0$ for some $i_{0}$, for a.e.\ $x\in X$ we have
\[
\lim_{N\to\infty} \E_{m\in [1,N]} | \E_{n\in [1,b(N)]} f_{1}(T_{1}^{m+p_{1}(n)}x) \dots f_{\ell}(T_{\ell}^{m+p_{\ell}(n)}x) |^{2}
= 0.
\]
\end{theorem}
Without the extra variable $m$ this result fails for general non-commuting transformations as can be seen from \cite[Theorem 1.7]{MR3043589}.
For commuting transformations convergence to zero for functions with zero $U^{k}(T_{i})$-seminorm (without the $m$) is known for distinct degree polynomials and is known to fail if two polynomials are rationally dependent.
The case of pairwise rationally independent polynomials and commuting transformations is open.

Next we apply our equidistribution result to the structured part of the decomposition provided by Theorem~\ref{thm:structure} to show that the rational Kronecker factor is in fact characteristic for our averages.
\begin{proposition}
\label{prop:Krat-char}
Let $p_{1},\dots,p_{\ell}$ be integer polynomials rationally independent from $1$ and let $(X,\mu),T_{1},\dots,T_{\ell},\Phi_{N}$ be as in Theorem~\ref{thm:recurrence}.
Let also $f_{i}\in L^{\infty}(X)$, $i=1,\dots,\ell$, and assume that $\E_{\mu}(f_{i_{0}} | \Krat(X,T_{i_{0}})) = 0$ for some $i_{0}$.
Then for a.e.\ $x\in X$ we have
\begin{equation}
\label{eq:lim0}
\lim_{N\to\infty} \E_{(m,n)\in \Phi_{N}} f_{1}(T_{1}^{m+p_{1}(n)}x) \dots f_{\ell}(T_{\ell}^{m+p_{\ell}(n)}x)
= 0.
\end{equation}
\end{proposition}
\begin{proof}
By symmetry we can assume $i_{0}=1$.
Let $\epsilon>0$ be fixed.
It suffices to show that the limit superior in \eqref{eq:lim0} is $O(\epsilon)$ outside a set of measure $O(\epsilon)$.

Let $k\in\N$ be as in Theorem~\ref{thm:Zk-char} and let $f_{i} = f_{i}^{s} + f_{i}^{e} + f_{i}^{u}$ be the decomposition provided by Theorem~\ref{thm:structure}.
By Theorem~\ref{thm:Zk-char} the contribution of the terms involving $f_{i}^{u}$ to the limit superior in \eqref{eq:lim0} vanishes a.e.

It follows from the growth assumption on $b(N)$ that
\[
\E_{(m,n)\in\Phi_{N}} | T_{i}^{m+p(n)}f_{i}^{e} | - \E_{m\leq N} | T_{i}^{m}f_{i}^{e} | \to 0
\text{ as } N\to\infty
\]
pointwise.
Since the functions $f_i^s$ and $f_i^e$ are bounded, this shows that the contribution of the terms involving $f_i^e$ to the limit superior in \eqref{eq:lim0} is $O(\epsilon)$ outside an exceptional set of measure $O(\epsilon)$ by the ergodic maximal inequality.

By Corollary~\ref{cor:structure-Krat} we may assume that $P(f_{1}(T_{1}^{n}x))=0$ for every $x$.
It suffices to show that in this case
\[
\limsup_{N}
|\E_{(m,n) \in \Phi_{N}} f_{1}^{s}(T_{1}^{m+p_{1}(n)}x) \dots f_{\ell}^{s}(T_{\ell}^{m+p_{\ell}(n)}x)|
= O(\epsilon).
\]
By the definition of a nilsequence there exist ergodic nilsystems $(Y_{i},S_{i})$, continuous functions $g_{i}\in C(Y_{i})$, and points $y_{i}\in Y_{i}$ such that $|f_{i}^{s}(T_{i}^{m}x) - g_{i}(S_{i}^{m}y_{i})| < \epsilon$ for every $m\in\Z$.
By Lemma~\ref{lem:limit-basic-nilseq-integral}(\ref{PKrat:ker}) we may assume
\begin{equation}
\label{eq:Pg1zero}
P(g_{1}(S_{1}^{m}y_{1})) = 0.
\end{equation}

By Lemma~\ref{lem:conn-comp-tot-erg} there exists a factorial $r$ such that $S_{i}^{r}$ is totally ergodic on each of the finitely many connected components of $Y_{i}$ for every $i$.

Let $(a,b)\in\Z^{2}$ be fixed, let $y_{i}' = S_{i}^{a+p_{i}(b)}y_{i}$, and let $Y'_{i} \subset Y_{i}$ be the connected component of $y_{i}'$ in $Y_{i}$.
It follows from Proposition~\ref{prop:equid} that the polynomial sequence
\[
(S_{1}^{m+p_{1}(n)}y_{1}, \dots, S_{\ell}^{m+p_{\ell}(n)}y_{\ell}),
\quad
(m,n)\in (a,b) + r \Z^{2}
\]
is well-distributed on $Y_{1}'\times\dots\times Y_{\ell}'$, and in particular
\[
\lim_{N\to\infty} \E_{(m,n) \in ((a,b)+r\Z^{2}) \cap \Phi_{N}} f_{1}^{s}(T_{1}^{m+p_{1}(n)}x) \dots f_{\ell}^{s}(T_{\ell}^{m+p_{\ell}(n)}x)
=
\prod_{i=1}^{\ell} \int_{Y_{i}'} g_{i}
+ O(\epsilon).
\]
By Lemma~\ref{lem:limit-basic-nilseq-integral}(\ref{PKrat:integral}) it follows from \eqref{eq:Pg1zero} that $\int_{Y_{1}'}g_{1} = 0$, and we obtain the claim by averaging over $a,b\in\{0,\dots,r-1\}$.
\end{proof}

\begin{proof}[Proof of Theorem~\ref{thm:recurrence}]
Let $f_{0}=f_{1}=\dots=f_{\ell}=1_{A}$ and let $\epsilon>0$ be fixed.
We will find a subgroup of $\Z^{2}$ on which the limit inferior of the ergodic averages exceeds $\mu(A)^{\ell+1}-\epsilon$.

Let $r\in\N$ be so large that $\|\E(f_{i}|\Krat(X,T_{i})) - \E(f_{i}|\K_{r}(X,T_{i})) \|_{2} <\epsilon/\ell$ for every $i$.
Consider the F\o{}lner sequence $\Phi_{N}' = [1,\lfloor N/r\rfloor]\times [1,\lfloor b(N)/r \rfloor]$.
Since $r\Phi_{N}' \subset \Phi_{N}$ and $|r\Phi_{N}'|/|\Phi_{N}| \to 1/r^{2} > 0$ as $N\to\infty$, it suffices to show that
\begin{equation}
\label{eq:power-bound}
\liminf_{N\to\infty} \int \E_{(m,n)\in r\Phi'_{N}} f_{0} T_{1}^{m+p_{1}(n)}f_{1} \dots T_{\ell}^{m+p_{\ell}(n)}f_{\ell}
\dif\mu
>
\mu(A)^{\ell+1}-\epsilon.
\end{equation}
Applying Proposition~\ref{prop:Krat-char} to the F\o{}lner sequence $(\Phi_{N}')$, measure-preserving transformations $T_{i}^{r}$, and polynomials $p_{i}(rn)/r$, we see that the limit inferior does not change upon replacing $f_{i}$ by $\E(f_{i}|\Krat(X,T_{i}))$ for each $i=1,\dots,\ell$ (strictly speaking, Proposition~\ref{prop:Krat-char} is only applicable to subsequences of $(\Phi_{N}')$ for which $\lfloor N/r\rfloor$ is strictly monotonically increasing, but since $(\Phi_{N}')$ can be split into $r$ such subsequences this suffices).
Hence, by the choice of $r$, the limit inferior in \eqref{eq:power-bound} is bounded below by
\[
\liminf_{N\to\infty} \int \E_{(m,n)\in r\Phi_{N}'} f_{0} T_{1}^{m+p_{1}(n)}\E(f_{1}|\K_{r}(X,T_{1})) \dots T_{\ell}^{m+p_{\ell}(n)}\E(f_{\ell}|\K_{r}(X,T_{\ell}))
\dif\mu
- \epsilon.
\]
Since $\E(f_{i}|\K_{r}(X,T_{i}))$ is $T_{i}^{m+p_{i}(n)}$-invariant for every $(m,n)\in r\Z^{2}$, this equals
\[
\int f_{0} \E(f_{1}|\K_{r}(X,T_{1})) \dots \E(f_{\ell}|\K_{r}(X,T_{\ell}))
\dif\mu
- \epsilon,
\]
and the above integral is bounded below by $\mu(A)^{\ell+1}$ by \cite[Lemma 1.6]{MR2794947}.
\end{proof}

\printbibliography
\end{document}
